\documentclass[11pt]{article}

\usepackage[font=small,labelfont=bf]{caption}

\usepackage{enumitem}
\usepackage{graphicx}
\usepackage{xcolor}
\usepackage{blkarray}
\usepackage{amsmath,amssymb}
\usepackage{bbm}
\usepackage{mathtools}
\usepackage{pdflscape}
\usepackage{float}
\usepackage{rotating}
\usepackage{titlesec}
\usepackage{amsthm}

\usepackage[
citestyle=numeric,
style=numeric,
uniquename=false,
sorting=nyt,
maxbibnames=8,
giveninits=true
]{biblatex}

\renewbibmacro{in:}{}

\newtheorem{theorem}{Theorem}[section]

\newtheorem{corollary}{Corollary}[theorem]

\titleformat{\paragraph}
{\normalfont\normalsize\bfseries}{\theparagraph}{1em}{}
\titlespacing*{\paragraph}
{0pt}{3.25ex plus 1ex minus .2ex}{1.5ex plus .2ex}

\DeclareMathOperator{\EX}{\mathbbm{E}}

\begin{document}
\setlength{\abovedisplayskip}{4pt}
\setlength{\belowdisplayskip}{10pt}
\setlength{\abovedisplayshortskip}{4pt}
\setlength{\belowdisplayshortskip}{10pt}

\title{Open networks of infinite server queues with non-homogeneous multivariate batch Poisson arrivals}

\author{Somya Mehra \and Peter G. Taylor}

\date{%
    \small{School of Mathematics and Statistics, The University of Melbourne, Parkville,  Australia}
}

\maketitle

\section*{Abstract}
In this paper, we consider the occupancy distribution for an open network of infinite server queues with multivariate batch arrivals following a non-homogeneous Poisson process, and general service time distributions. We derive a probability generating function for the transient occupancy distribution of the network, and prove that it is necessary and sufficient for ergodicity that the expected occupancy time for each batch be finite. Further, we recover recurrence relations for the transient probability mass function formulated in terms of a distribution obtained by compounding the batch size with a multinomial distribution.

\section{Introduction}

Infinite server queues with Poisson batch arrivals --- denoted $\text{M}_t^\text{X}/\text{G}/\infty$ queues in Kendall's notation --- have elicited significant research attention, from the canonical work of \textcite{shanbhag1966infinite, reynolds1968some, brown1969some} in the 1960s, to studies by \textcite{chatterjee1989non, liu1991thegr} and the recent treatise of \textcite{daw2019distributions}. The  independence between customers within the queue --- termed the `non-interacting property' by \parencite{liu1991thegr} --- allows for an analytic characterisation of the numbers of busy servers and departures, typically in the form of a probability generating function (PGF) \parencite{brown1969some, chatterjee1989non, daw2019distributions}.\\

A natural generalisation is an open network of infinite server queues, with multivariate batch arrivals following a non-homogeneous Poisson process and general service time distributions. With customers arriving one at a time, networks of $\text{M}_t/\text{G}/\infty$ queues have been studied by \textcite{harrison1981note, keilson1990networks, massey1993networks}, amongst others. In contrast, networks with multivariate batch arrivals have received comparatively little attention. Models have been proposed for specific applications such as private-line telecommunication services \parencite{mccalla2002time} and malarial relapses \parencite{mehra2022hypnozoite}, while stochastic fluid networks have been shown to arise as scaling limits of appropriate batch-arrival infinite server queueing networks \parencite{kella1999linear}. To the best of our knowledge, results for more general networks of $\cdot/G/\infty$ queues with multivariate batch arrivals are yet to appear in the literature.\\

Here, we analyse the occupancy distribution in a network of infinite server queues, with multivariate batches arriving according to a non-homogeneous Poisson process. Our model assumptions are set out in Section \ref{sec::model_structure}. We derive the PGF for the transient queue length distribution in Theorem \ref{theorem::open_network_queue_pgf}, extending the construction we adopted in \textcite{mehra2022hypnozoite}; similar arguments have been presented previously by \parencite{brown1969some, harrison1981note, chatterjee1989non} and others.\\

While infinite server queues with single arrivals are necessarily stable when the expected service time is finite, the queue length Markov chain can be transient or null-recurrent when batch arrivals are permitted. Stability conditions for the $\text{M}^\text{X}/\text{G}/\infty$ queue, accommodating heterogenous customers within each batch, have been characterised by \textcite{cong1994mx}. When service times are exponential, \textcite{yajima2016stability} have shown that it is necessary and sufficient for an infinite server queue with a batch Markovian arrival process to be stable that the batch size distribution has a bounded logarithmic moment. For networks with multivariate batch arrivals, we extend these results in Section \ref{sec::stability_cond} to show that a necessary and sufficient condition for ergodicity is that the expected occupancy time for each batch be finite (Theorem \ref{sec::stability_cond}). When the customer occupancy time distributions have exponential tails, this is equivalent to the logarithmic moment condition of \parencite{cong1994mx, yajima2016stability} (Corollaries \ref{corollary::logarthmic_moment} and \ref{corollary::logarthmic_moment_2}).\\

We recover recurrence relations for the probability mass function (PMF) of the time-dependent queue length distribution, formulated with respect to the PMF of the batch size compounded with a multinomial distribution in Theorem \ref{theorem:recurrence_relation}. The utility of these formulae is constrained by the underlying compound distribution. In the context of the $\text{M}^\text{X}/\text{G}/\infty$ queue, \textcite{willmot2001transient, willmot2002transient, willmot2009time} found that for distributions within `Sundt and Jewell's family' --- encompassing the binomial, Poisson, logarithmic, geometric and negative binomial distributions, and zero-inflated analogues thereof \parencite{willmot1988sundt} --- the recurrence relations may be computationally tractable. Here, we find that univariate batch sizes within `Sundt and Jewell's family' \parencite{willmot1988sundt} may likewise yield tractable formulae when the initial queue allocation for each customer is independent and identically distributed (i.i.d.) (Section \ref{sec::sundt_jewel_batch}).

\section{Model structure} \label{sec::model_structure}

Consider an open network of $J$ infinite server queues, indexed by $j = 1, \dots, J$, such that
\begin{itemize}
    \item Arrivals occur at points $T_1,T_2,\ldots $ of a non-homogeneous Poisson process with rate $\lambda(t)$.
    \item At time point $T_i$, for $j=1,\ldots,J$, $\Sigma_{ij}$ customers arrive at queue $j$. The vectors $\mathbf{\Sigma}_i = (\Sigma_{i1},\ldots,\Sigma_{iJ})$ are independent vector-valued random variables with non-negative integer components selected from a multivariate distribution with probability generating function
    \begin{align}
        G_{\mathbf{S}}(z_1, \dots, z_J) := \EX \Bigg[ \prod^J_{j=1} z_j^{S_{j}} \Bigg]
    \end{align}
    defined on a subset ${\cal D} \subseteq {\cal C}^J$ that contains the $J$-dimensional unit polydisc with $|z_j|\leq1$. Note that the distribution of $\mathbf{\Sigma}_i$ does not depend on $i$. We use the generic random variable $\mathbf{S}$ to denote a random variable with this distribution.
    \item The $\ell^{th}$ customer that enters queue $j$ at time point $T_i$ is marked with a stochastic process $X_{i,\ell}^j(t)$ where $X_{i,\ell}^j(t) = k$ if the customer will be in queue $k$ at time $T_i+t$. For each $i$ and $\ell = 1,\ldots,\Sigma_{ij}$, the $X_{i,\ell}^j(t)$ are selected independently with a law such that $P(X_{i,\ell} ^j(t) = k) = q^j_k(t)$.
\end{itemize}

Under this model, the routes of customers through the network (that is, the sequence of visited nodes and associated service times) are mutually independent, conditional on the node of entry. This `non-interacting property' \parencite{liu1991thegr} enables explicit analysis. Here, we focus on the number of customers in the network at time $t$, denoted by the random vector $\mathbf{N}(t) = (N_1(t), \dots, N_J(t))$.\\

For notational convenience, we introduce the random vector $\mathbf{C}_i(t) = (C_{i,1}(t),\ldots,C_{i,J}(t))$, with
\begin{align}
    C_{i,k} (t) = \sum_{j=1}^J \sum_{\ell = 1}^{\Sigma_{ij}} \mathbbm{1} \{ X_{i,\ell} ^j(t) = k \} \label{eq:cj}
\end{align}
giving the number of customers from the $i$th batch that are in queue $k$ at time $T_i+t$. It follows from our assumptions above that the distribution of $\mathbf{C}_i(t)$ is independent of $i$ and we denote by $\mathbf{C}(t)$ a generic random variable with this distribution.

\section{The transient queue length distibution}

We first characterise the time-dependent PGF for the state of the network, generalising the construction of \textcite{mehra2022hypnozoite} (which was used to analyse an open network of infinite server queues tailored to within-host superinfection and hypnozoite dynamics for \textit{Plasmodium vivax} malaria).

\begin{theorem}{(Multivariate PGF)} \label{theorem::open_network_queue_pgf}
For $z=(z_1,\ldots,z_J) \in {\cal D}$
\begin{align}
    \phi_t(\mathbf{z}) &:= \EX \bigg[ \prod^J_{k=1} z_j^{N_j(t)} \bigg]\\
    &= \exp \bigg\{ -\int^t_0 \lambda(\tau) \bigg[1 - G_\mathbf{S} \bigg( 1 + \sum^J_{k=1} (z_k - 1) q^1_k(t), \dots, 1 + \sum^J_{k=1} (z_k - 1) q^J_k(t)  \bigg) \bigg] d \tau \bigg\}. \label{eq::queue_pgf_general}
\end{align} 
\end{theorem}

\begin{proof}
We use similar reasoning to \textcite{mehra2022hypnozoite}. In short, we condition on the multivariate batch size, and the sequence of arrival times. For related systems, similar proofs have been previously presented by others, including \textcite{brown1969some, harrison1981note, chatterjee1989non}.\\\

We begin by deriving the PGF for the generic random vector $\mathbf{C}(t)$ described in Equation (\ref{eq:cj}). Since $X^j_{i, \ell}(t)$ are independent random variables taking values in $\{1,\ldots,J\}$, the number of customers originating in queue $j$ that are in queue $k$ at time $T_i+t$ has a multinomial distribution with parameters $q^j_k(t)$. Therefore, conditional on a multivariate batch of size $\Sigma_{i1} = n_1, \dots, \Sigma_{iJ} = n_J$,
\begin{align*}
    \EX \bigg[ \prod^J_{k=1} z_k^{C_k(t)}  \, \Big| \, \Sigma_{i1} = n_1, \dots, \Sigma_{iJ} = n_J \bigg] =  \prod^J_{j=1} \bigg( 1 + \sum^J_{k=1} (z_k - 1) q^j_k(t) \bigg)^{n_k}.
\end{align*}

By the law of total expectation, it follows that
\begin{align}
    \EX \Bigg[ \prod^J_{k=1} z_k^{C_k(t)} \Bigg] &= \sum^\infty_{n_1 = 0} \dots \sum^\infty_{n_J = 0} P(\Sigma_{i1} = n_1, \dots, \Sigma_{iJ} = n_J) \EX \Bigg[ \prod^J_{k=1} z_k^{C_k(t)}  \, \Big| \, \Sigma_{i1} = n_1, \dots, \Sigma_{iJ} = n_J \Bigg] \notag \\
    &= G_\mathbf{S} \bigg( 1 + \sum^J_{k=1} (z_k - 1) q^1_k(t), \dots, 1 + \sum^J_{k=1} (z_k - 1) q^J_k(t)  \bigg). \label{eq::compound_queue_PGF}
\end{align}

Under the assumption that arrivals follow a non-homogeneous Poisson process, the number of arrivals $M(t)$ in the interval $[0, t)$
\begin{align*}
    M(t) \sim \text{Poisson}(\Lambda(t))
\end{align*}
where
\begin{align*}
    \Lambda(t) = \int^t_0 \lambda(\tau) d \tau.
\end{align*}

Following \textcite{harrison1981note}, conditional on the event $\{M(t) = m\}$, the distribution of $T_1, \dots, T_m$ is the same as the distribution of the order statistics of $m$ i.i.d. variables with density
\begin{align*}
    f(\tau) = \frac{\lambda(\tau)}{\Lambda(t)} \mathbbm{1}_{\{ \tau \in [0, t) \}}, 
\end{align*}
yielding the conditional PGF
\begin{align*}
    \EX \bigg[ &\prod^J_{k=1} z_k^{N_k(t)} \Big| M(t) = m \bigg] = \Bigg( \int^t_0 \frac{\lambda(\tau)}{\Lambda(t)} \cdot \EX \Bigg[ \prod^J_{k=1} z_k^{C_k(t - \tau)} \Bigg] d \tau \Bigg)^m.
\end{align*}

Using the law of total expectation, we thus deduce
\begin{align}
    \EX \bigg[ \prod^J_{j=1} z_j^{N_j(t)} \ \bigg] &=\sum^\infty_{m=0} \EX \bigg[ \prod^J_{j=1} z_j^{N_j(t)} \Big| M(t) = m \bigg] \cdot P(M(t) = m) \notag \\
    &= \exp \bigg\{ -\int^t_0 \lambda(\tau) \Bigg(1 -  \EX \Bigg[ \prod^J_{j=1} z_j^{C_j(t - \tau)} \Bigg] \Bigg) d \tau \bigg\}. \label{eq::queue_pgf_general_compound_dist}
\end{align}

Substituting Equation (\ref{eq::compound_queue_PGF}) into Equation (\ref{eq::queue_pgf_general_compound_dist}) yields expression (\ref{eq::queue_pgf_general}).
\end{proof}

\section{Necessary and sufficient conditions for ergodicity} \label{sec::stability_cond}

Here, we characterise the stability, that is, ergodicity, of the network, assuming a homogeneous arrival rate $\lambda(t) = \lambda$. For $\text{M}/\text{G}/\infty$ queues, a necessary and sufficient condition for ergodicity is that the occupation time distribution has finite expectation. The analogous constraint in this setting is that at least one customer from each batch will be present somewhere in the network for a time $W$ that has finite expectation.\\

A customer who arrives to queue $j$ at time $T$ has left the network the network before time $T+t$ with probability
\begin{align*}
    1 - Q_j(t) = 1 - \sum^J_{k=1} q^j_k(t).
\end{align*}

The expected time that at least one customer from a batch is present in the network can then be written
\begin{align}
    \EX[W] &= \sum^{\infty}_{n_1=0} \dots \sum^{\infty}_{n_J=0} P(\mathbf{S}=(n_1, \dots, n_J)) \cdot \int^\infty_0 \bigg( 1 - \prod^J_{j=1} \big[ 1 - Q_j(t) \big]^{n_j} \bigg) d \tau \notag \\
    &= \int^\infty_0 \Big[1 - G_\mathbf{S} \Big( 1 - Q_1(\tau), \dots, 1 - Q_J(\tau)  \Big) \Big] d \tau. \label{eq:exp_service_time}
\end{align}
where we have used Tonelli's theorem to swap the order of summation and integration; and noted that the resultant integrand can be written as a sum of two absolutely convergent series. This expression is central to the ergodicity of the network.\\

\begin{theorem} \label{theorem::ergodicity}
The network is ergodic if and only if $\EX[W] < \infty$  
\end{theorem}
\begin{proof}
Let $\Omega \in \mathcal{C}^J$ denote the open $J$-dimensional polydisc with $|z_j| < 1$ for each $j \in \{1, \dots, J\}$. By the dominated convergence theorem, for all $\mathbf{z} \in \Omega$, 
\begin{align*}
    \lim_{t \to \infty}  \EX \bigg[ \prod^J_{k=1} z_k^{N_j(t)} \bigg] = \EX \bigg[ \prod^J_{k=1} z_k^{\lim_{t \to \infty} N_j(t)} \bigg]
\end{align*}
since $P(\mathbf{N}(t) = \mathbf{n}) \leq 1$ for all $\mathbf{n} \in \mathbbm{Z}^J, t \geq 0$.\\

Observe that the transient PGF $\phi_t(\mathbf{z})$ for $\mathbf{N}(t)$ defines a family of functions $\{\phi_t\}_{t \geq 0}$ that is both analytic and bounded on $\Omega$, with $|\phi_t(\mathbf{z})| \leq 1$ for all $t \geq 0$. By Montel's theorem, there exists a subsequence $\phi_{t_i}$ which converges uniformly on all compact subsets $K \subset \Omega$ (Theorem 1.7.3 of \parencite{korevaar2017several}). It follows from a theorem of Weirstrass (Theorem 1.7.1 of \parencite{korevaar2017several}) that $\lim_{t \to \infty} \phi_t$ is analytic on $\Omega$, with the uniform convergence of partial derivatives
\begin{align*}
    \lim_{i \to \infty} \Bigg\{ \frac{\partial^{v_1 + \dots + v_n}}{\prod^n_{k=1} \partial z_k^{v_k}} \phi_{t_i}(\mathbf{z}) \Bigg\} =  \frac{\partial^{v_1 + \dots + v_n}}{\prod^n_{k=1} \partial z_k^{v_k}} \Big\{ \lim_{t \to \infty} \phi_t(\mathbf{z}) \Big\}
\end{align*}
on all compact subsets $K \subset \Omega$. In particular, for all $\mathbf{v} \in \mathbbm{Z}^J$, we note that
\begin{align*}
    \lim_{i \to \infty} P( \mathbf{N}(t_i) = \mathbf{v}) = \frac{\partial^{v_1 + \dots + v_n}}{\prod^n_{k=1} \partial z_k^{v_k}} \Big\{ \lim_{t \to \infty}  \phi_t(\mathbf{z}) \Big\} \Big|_{\mathbf{z} = \mathbf{0}}.
\end{align*}


It is therefore sufficient to consider the infinite time limit of the multivariate PGF (\ref{eq::queue_pgf_general})
\begin{align*}
    \lim_{t \to \infty} \phi_t(\mathbf{z}) & = \exp \bigg\{ - \lambda \int^\infty_0 \bigg[1 - G_\mathbf{S} \bigg( 1 + \sum^J_{k=1} (z_k - 1) q^1_k(\tau), \dots, 1 + \sum^J_{k=1} (z_k - 1) q^J_k(\tau)  \bigg) \bigg] d \tau \bigg\}.
\end{align*}
As in \textcite{cong1994mx}, it is necessary and sufficient for the limiting function to be a PGF that the multivariate limit
\begin{align*}
    \lim_{z_k \uparrow 1, \, k \in \{1, \dots, J \}} \Big\{ \lim_{t \to \infty}  \phi_t(\mathbf{z}) \Big\} = 1,
\end{align*}
where $(z_1, \dots, z_J) \to (1, \dots, 1)$ along any path through $\Omega$. Since the generating function is analytic inside $\Omega$, this is equivalent to the condition
\begin{align}
    \lim_{z \uparrow 1} \int^\infty_0 \Big[1 - G_\mathbf{S} \Big( 1 + (z-1) Q_1(\tau), \dots, 1 + (z-1) Q_J(\tau)  \Big) \Big] d \tau \bigg\} = 0. \label{eq:stability_cond}
\end{align}

For notational convenience, we denote the integrand
\begin{align*}
    H(z, \tau) := 1 - G_\mathbf{S} \big( 1 + (z-1) Q_1(\tau), \dots, 1 + (z-1) Q_J(\tau) \big),
\end{align*}
and observe from Equation (\ref{eq:exp_service_time}) that
\begin{align*}
    \EX[W] = \int^\infty_0 H(0, \tau) d \tau.
\end{align*}
Since $G_S$ is a multivariate PGF, for fixed $\tau$, $H(z, \tau)$ is a decreasing function of $z$ in the domain $[0, 1]$. In particular, for each $\tau \geq 0$ we have pointwise convergence $H(z, \tau) \to 0$ as $z \to 1$, with $|H(z, \tau)| \leq H(0, \tau)$. Suppose $\EX[W] < \infty$. Then by the dominated convergence theorem,
\begin{align*}
    \lim_{z \uparrow 1} \int^\infty_0 H(z, \tau) d \tau =  \int^\infty_0 \lim_{z \uparrow 1} H(z, \tau) d \tau = 0.
\end{align*}
We thus see that $\EX[W]<\infty$ is a sufficient condition for Equation (\ref{eq:stability_cond}) to hold.\\

Now, suppose that the network is ergodic. Then the limiting probability of the queue being empty can be written
\begin{align*}
    \lim_{t \to \infty} P(\mathbf{N}(t) = \mathbf{0}) = e^{- \lambda \EX[W]} > 0,
\end{align*}
which necessarily implies $\EX[W]<\infty$.
\end{proof}

In direct analogy to Lemma 1 of \textcite{cong1994mx}, under the constraint that the time spent in the network by each customer has finite expectation --- which is necessary for batch occupancy time to have finite expectation, that is, $\EX[W] < \infty$  --- we recover a sufficient condition for ergodicity.

\begin{corollary}
Suppose the network occupation time distributions $Q_j(t)$ have finite expectation. Then the network is ergodic if $\EX[S_1 + \dots + S_J ]<\infty$.
\end{corollary}

\begin{proof}
Similarly to \textcite{cong1994mx}, we observe that
\begin{align*}
    1- \prod^J_{j=1} [1-Q_j(\tau)]^{n_j} \leq \sum^J_{j=1} n_j Q_j(\tau)
\end{align*}
Therefore, from Equation (\ref{eq:exp_service_time}), it follows that
\begin{align*}
    \EX[W] \leq \sum^\infty_{n_1=0} \dots \sum^\infty_{n_J=0} P(\mathbf{S} = (n_1, \dots, n_J)) \int^\infty_0 \sum^J_{j=1} n_j Q_j(\tau) d \tau = \sum^J_{j=1} \EX[S_j] \cdot \int^\infty_0 Q_j(\tau) d \tau.
\end{align*}
A finite expected network occupation time for each customer $\int^\infty_0 Q_j(\tau) d \tau<\infty$ and a finite mean batch size $\EX[S_j] < \infty$ for each $j \in \{1, \dots, J\}$ is thus a sufficient condition for $\EX[W]<\infty$ and, by Theorem \ref{theorem::ergodicity}, the ergodicity of the network. 
\end{proof}

For $\text{M}^\text{X}/\text{M}/\infty$ queues, \textcite{cong1994mx, yajima2016stability} have shown that a necessary and sufficient condition for stability is the batch size $S$ has a finite logarithmic moment, that is,
\begin{align*}
    \EX[ \log(S+1) ] < \infty.
\end{align*}
Given an exponential service time of mean duration $\mu$ per customer, the expected occupancy time for a batch can written
\begin{align*}
    \EX[W] = \sum^\infty_{j=1} P(S=j) \sum^j_{k=1} \frac{1}{\mu k}.
\end{align*}
Since
\begin{align}
    \sum^n_{j=1} \frac{1}{j} - \log(n + 1) \to \gamma_e \text{ as } n \to \infty \label{eq:harmonic}
\end{align}
where $\gamma_e$ denotes the Euler–Mascheroni constant \parencite{whittaker2020course}, in the case of the $\text{M}^\text{X}/\text{M}/\infty$ queue
\begin{align*}
    \EX[ \log(S+1) ] < \infty \iff \EX[W] < \infty.
\end{align*}

We can generalise this observation to recover a sufficient condition for ergodicity when the network occupancy time distributions $Q_j(t)$ have exponentially-bounded tails, that is, there exists $\delta>0, t_0>0$ such that
\begin{align*}
    Q_j(\tau) \leq e^{-\delta \tau} \text{ for all } \tau \geq t_0 \text{ and } j \in \{1, \dots, J \}.
\end{align*}

\begin{corollary}
\label{corollary::logarthmic_moment}
Suppose the occupancy time distributions $Q_j(t)$ have exponentially bounded tails. Then the network is ergodic if $\EX[ \log(S_1 + \dots + S_J + 1)]<\infty$.
\end{corollary}
\begin{proof}
By assumption, there exist $\delta > 0, t_0>0$ such that 
\begin{align*}
    0 < G_\textbf{S} \big(1 - e^{-\delta \tau}, \dots, 1-e^{-\delta \tau} \big) \leq G_\textbf{S} \big(1 - Q_1(\tau), \dots, 1-Q_J(\tau) \big) \leq 1 \text{ for all } \tau \geq t_0
\end{align*}
since $G_\textbf{S}$ is a generating function. Equation (\ref{eq:exp_service_time}) then yields the bound
\begin{align*}
    \EX[W] \leq t_0 + \int^\infty_{t_0} \Big[ 1 - G_\textbf{S} \big(1 - e^{-\delta \tau}, \dots, 1-e^{-\delta \tau} \big) \Big] d \tau
\end{align*}

Setting $S = S_1 + \dots + S_J$, we observe that
\begin{align*}
    G_\textbf{S} \big(1 - e^{-\delta \tau}, \dots, 1-e^{-\delta \tau} \big) = \sum^\infty_{n=0} P(S=n) (1-e^{-\delta \tau})^n.
\end{align*}

Therefore, a \textit{sufficient} condition for $\EX[W]<\infty$ --- and, by Theorem \ref{theorem::ergodicity}, the ergodicity of the network --- is that
\begin{align*}
    Y:= \int^\infty_0 \sum^\infty_{n=0} P(S=n) \big[ 1 - (1-e^{-\delta \tau})^n \big] d \tau < \infty.
\end{align*}
Using the binomial expansion and identity 0.155.4 of \parencite{jeffrey2007table}, we compute.
\begin{align*}
    Y &= \sum^\infty_{n=1} P(S=n) \sum^n_{k=1} {n \choose k} \frac{(-1)^{k+1}}{\delta k} = \sum^{\infty}_{n=1} P(S=n) \sum^n_{k=1} \frac{1}{\delta k},
\end{align*}
so from Equation (\ref{eq:harmonic}) \parencite{whittaker2020course}
\begin{align*}
    Y < \infty \iff \EX[\log(S+1)] < \infty.
\end{align*}
Therefore, $\EX[ \log(S + 1)] < \infty$ is a sufficient condition for the ergodicity of the network.
\end{proof}

For Markovian queueing networks, where the time spent by a customer is each node is exponentially-distributed, we can strengthen Corollary \ref{corollary::logarthmic_moment} to recover a necessary and sufficient condition for ergodicity.

\begin{corollary}
For each $j \in \{1, \dots, J\}$, suppose there exists $\delta_j>0$ such that $Q_j(t) = \Theta(e^{-\delta_j \tau})$ in the limit $\tau \to \infty$. Then the network is ergodic if and only if $\EX[\log(S_1 + \dots S_J + 1)] < \infty$. \label{corollary::logarthmic_moment_2}
\end{corollary}

\begin{proof}
Sufficiency follows directly from Corollary \ref{corollary::logarthmic_moment}. By assumption, there exists $\eta>\delta_j$ and $t_0 > 0 $ such that
\begin{align*}
    Q_j(\tau) \geq e^{-\eta \tau} \text{ for all } \tau \geq t_0 \text{ and } j \in \{1, \dots, J \}
\end{align*}
Set $S=S_1 + \dots + S_J$. Adopting much the same reasoning as the proof of Corollary \ref{corollary::logarthmic_moment}, we obtain a lower bound for the expected batch occupancy time
\begin{align*}
    \EX[W] + t_0 \geq \int^\infty_{0}  \Big[ 1 - G_\mathbf{S} \big( 1-e^{-\eta \tau}, \dots, 1-e^{-\eta \tau} \big) \Big] d \tau = \sum^\infty_{n=1} P(S=n) \sum^n_{k=1} \frac{1}{\eta k}.
\end{align*}
From Equation (\ref{eq:harmonic}) \parencite{whittaker2020course}, we observe that
\begin{align*}
    \EX[\log(S + 1)] = \infty \implies \sum^\infty_{n=1} P(S=n) \sum^n_{k=1} \frac{1}{\eta k} = \infty \implies \EX[W] = \infty
\end{align*}
so by Theorem \ref{theorem::ergodicity},  $\EX[\log(S+1)] < \infty$ is a necessary condition for ergodicity.
\end{proof}

In a similar vein, we can recover a sufficient condition for ergodicity in the case of fat-tailed customer occupancy time distributions.

\begin{corollary}
Suppose there exists $\alpha>1$ and $t_0>0$ such that $Q_j(t) \leq t^{-\alpha}$ for all $j \in \{1, \dots, J \}$ and $t \geq t_0$. Then the network is ergodic if $\EX[(S_1 + \dots + S_J)^{1/\alpha}] < \infty$. 
\end{corollary}
\begin{proof}
Using similar reasoning to Corollary \ref{corollary::logarthmic_moment}, we obtain the lower bound 
\begin{align*}
    \EX[W] \leq t_0 + \int^\infty_{t_0} \sum^\infty_{n=0} P(S=n) \big[ 1 - (1-t^{-\alpha})^n \big] d \tau.
\end{align*}
Performing a change of variables $u = 1-t^{-\alpha}$, we find that the condition
\begin{align}
    Z:= \sum^\infty_{n=0} P(S=n) \int^1_0 \Bigg[ \sum^{n-1}_{i=0} (1-u)^{-\frac{1}{\alpha}} u^i \Bigg] du < \infty \label{eq:fat_tail_intermediate_1}
\end{align}
is sufficient to ensure $\EX[W]<\infty$, where the interchanging of the order of summation and integration is justified by Tonelli's theorem. Using the integral representation of the beta function and identity 8.384.1 of \parencite{jeffrey2007table}, we compute
\begin{align}
     \int^1_0 \Bigg[ \sum^{n-1}_{i=0} (1-u)^{-1/\alpha}  u^i  \Bigg]d u = \sum^{n-1}_{k=0} B \Big(1-\frac{1}{\alpha}, k+1 \Big) = \Gamma \Big( 1 - \frac{1}{\alpha} \Big) \sum^{n-1}_{k=0} \frac{\Gamma(k+1)}{\Gamma \Big( k + 2 - \frac{1}{\alpha} \Big)}. \label{eq:fat_tail_intermediate_2}
\end{align}
Substituting Equation (\ref{eq:fat_tail_intermediate_2}) in  (\ref{eq:fat_tail_intermediate_1}) and interchanging the order of summation (which is justified by Tonelli's theorem) yields
\begin{align*}
    Z = \Gamma \Big( 1 - \frac{1}{\alpha} \Big) \sum^{\infty}_{k=0} \frac{\Gamma(k+1)}{\Gamma \Big( k + 2 - \frac{1}{\alpha} \Big)} P(S \geq k + 1).
\end{align*}
By identity 6.1.46 of \parencite{abramowitz1988handbook}
\begin{align*}
    \frac{\Gamma(k+1)}{\Gamma \Big( k + 2 - \frac{1}{\alpha} \Big)} \sim k^{\frac{1}{\alpha}-1} \text{ in the limit } k \to \infty.
\end{align*}
Therefore,
\begin{align*}
    Z < \infty \iff  \sum^{\infty}_{k=1} k^{\frac{1}{\alpha}-1} P(S \geq k + 1) = \sum^{\infty}_{n=2} P(S=n) \sum^{n-1}_{k=1} k^{\frac{1}{\alpha}-1}  < \infty.
\end{align*}
where we have interchanged the order of summation. Noting that
\begin{align*}
    \sum^{n}_{k=1} k^{\frac{1}{\alpha}-1} \sim \frac{1}{\alpha} n^\frac{1}{\alpha} \text{ in the limit } n \to \infty
\end{align*}
(identity 0.121 of \parencite{jeffrey2007table}), we further deduce that
\begin{align*}
    \EX\big[ S^\frac{1}{\alpha} \big] < \infty \iff \sum^{\infty}_{n=1} P(S=n) \sum^n_{k=1} k^{\frac{1}{\alpha}-1}  < \infty.
\end{align*}
Therefore, $\EX \big[ S^\frac{1}{\alpha} \big] < \infty \implies \EX[W]<\infty$ and the claim follows from Theorem \ref{theorem::ergodicity}.
\end{proof}

\section{Recurrence relations for the PMF of the transient occupancy distribution} \label{sec::recurrence_relations}

For a single batch arrival ($\text{M}^\text{X}/\text{G}/\infty$) queue with a homogeneous arrival process, \textcite{willmot2001transient, willmot2002transient, willmot2009time} recognised the PGF governing the number of customers at a fixed time $t$ as that of a compound Poisson distribution. They observed that application of the Panjer-Adleson recursion scheme \parencite{panjer1981recursive, tijms2003first} yields recurrence relations for the PMF of the time-dependent queue length distribution, formulated with respect to the PMF of a distribution obtained by compounding the batch size with a binomial distribution. Using the transient PGF given by Equation (\ref{eq::queue_pgf_general}), we can likewise recover a recurrence relation for the PMF of the time-dependent occupancy distribution, formulated in terms of the PMF of $\mathbf{C}(t)$ (Equation (\ref{eq:cj})), using an analogous argument to \textcite{panjer1981recursive}.

\begin{theorem}{(A recurrence relation for the multivariate PMF.)} \label{theorem:recurrence_relation}\newline
For $v \in \{2, \dots, J \}$ such that $n_v \geq 1$,
\begin{align}
    P \big[ & \mathbf{N}(t) = (n_1, \dots, n_v, 0, \dots 0) \big] \notag\\
    & = \sum^{n_1}_{i_1 = 0} \dots \sum^{n_{v-1}}_{i_{v-1} = 0} \sum^{n_{v}}_{i_{v} = 1} \Bigg[ \frac{i_v}{n_v} \cdot  P \big[ \mathbf{N}(t) = (n_1-i_1, \dots, n_v-i_v, 0, \dots, 0) \big] \notag \\
    & \qquad \qquad \cdot \int^t_0 \lambda(\tau) \cdot P \big[ \mathbf{C}(t-\tau) = ( i_1, \dots, i_v, 0, \dots, 0) \big] d \tau \Bigg]. \label{eq::transient_recurrence_relation}
\end{align}
\end{theorem}

\begin{proof}
Following the proof of Theorem \ref{theorem::open_network_queue_pgf}, given that there are $M(t)=m$ arrivals in the interval $[0, t)$, the PMF for the number of customers in each queue $1, \dots, J$ at time $t$ attributable to each arrival event is i.i.d. with PMF
\begin{align*}
    f(i_1, \dots, i_J) = \int^t_0 \frac{\lambda(\tau)}{\Lambda(t)} P \big[ \mathbf{C}(t-\tau) = (i_1, \dots, i_J) \big] d \tau.
\end{align*}

Denote by $f^{(m)}$ the $m$-fold convolution of $f$. Then by the law of total probability, the PMF of $\mathbf{N}(t)$ can be written in the form
\begin{align}
    P \big[ \mathbf{N}(t) = (n_1, \dots, n_J) \big] &= \sum^\infty_{m=0} \frac{e^{-\Lambda(t)} \Lambda(t)^m}{m!} f^{(m)}(n_1, \dots, n_J). \label{eq::compound_poisson_pmf}
\end{align}
Now, by symmetry,
\begin{align}
    \sum^{n_1}_{i_1=0} \dots \sum^{n_J}_{i_J=0} \frac{i_v}{n_v} f(i_1, \dots, i_J) \cdot f^{(m-1)}(n_1-i_1, \dots, n_J-i_J) = \frac{f^{(m)}(n_1, \dots, n_J)}{m}; \label{eq::symmetry_convolution}
\end{align}
that is, given that the sum of $m$ i.i.d. random vectors, each with PMF $f$, is $(n_1, \dots, n_J)$, the conditional mean of the $v^\text{th}$ element is $n_v/m$ (Relation II of \textcite{panjer1981recursive}).\\

Using Equations (\ref{eq::compound_poisson_pmf}) and (\ref{eq::symmetry_convolution}), for $(n_1, \dots, n_J) \neq \mathbf{0}$, we obtain the expression
\begin{align}
     P \big[ \mathbf{N}(t) = (n_1, \dots, n_J) \big] &=  \sum^{n_1}_{i_1 = 0} \dots \sum^{n_J}_{i_J=0} \frac{i_v}{n_v} f(i_1, \dots, i_j) \sum^\infty_{m=1} \frac{e^{-\Lambda(t)} \Lambda(t)^{m}}{(m-1)!} f^{(m-1)}(n_1 - i_1, \dots, n_J - i_J) \notag \\
     &= \sum^{n_1}_{i_1 = 0} \dots \sum^{n_J}_{i_J=0} \frac{i_v}{n_v} \Lambda(t) f(i_1, \dots, i_j) P \big[ \mathbf{N}(t) = (n_1-i_1, \dots, n_J-i_J) \big] \label{eq::prelim_recurrence_relation}
\end{align}
where interchanging the order of summation is justified since the series is absolutely convergent. Sequential application of Equation (\ref{eq::prelim_recurrence_relation}) along the lower triangular integer lattice yields precisely the recurrence relation (\ref{eq::transient_recurrence_relation}). The expression (\ref{eq::transient_recurrence_relation}) can also be obtained by computing partial derivatives of the multivariate PGF (\ref{eq::queue_pgf_general}), using the formulae provided in \textcite{miatto2019recursive}.
\end{proof}

Through an application of Faa di Bruno's formula, we noted in \textcite{mehra2022hypnozoite} that the transient distribution for the $\text{M}_\text{t}^\text{X}/\text{G}/\infty$ queue can also be formulated in terms of Bell polynomials \parencite{comtet2012advanced}, which yield an analogous recurrence structure. For queueing networks with multivariate batch arrivals, alternative representations of the transient PMF in terms of multivariate Bell polynomials \parencite{schumann2019multivariate} can also be recovered.\\

Likewise, we can derive recurrence relations for the joint factorial moments of $\mathbf{N}(t)$, formulated with respect to those of $\mathbf{C}(t)$.

\subsection{Some special cases}

The recurrence relations derived in Theorem \ref{theorem::open_network_queue_pgf} are formulated in terms of the PMF for the distribution of $\mathbf{C}(t)$. The tractability of these formulae are therefore constrained by the multivariate PMF for $\mathbf{C}(t)$. Here, we provide explicit formulae for several special cases.

\subsubsection{Univariate batch sizes} \label{sec::sundt_jewel_batch}

For $\text{M}^\text{X}/\text{G}/\infty$ queues, \textcite{willmot2001transient, willmot2002transient} observed that batch sizes within `Sundt and Jewell's family of discrete distributions' \parencite{sundt1981further, willmot1988sundt} may yield tractable formulae. This class of distributions has a probability mass function that is characterised by a recurrence relation of the form
\begin{align*}
    P(S=n) = P(S=n-1) \cdot \Big( a + \frac{b}{n} \Big)
\end{align*}
for fixed $a, b$. It was shown in in Theorem 1 of \textcite{willmot1988sundt} that this class includes Poisson, negative binomial and logarithmic distributions, and zero-inflated analogues thereof.\\

In the context of queueing networks, we restrict our attention to multvariate batch size distributions $\mathbf{S}$ that can be expressed in the form
\begin{align*}
    \EX \Bigg[ \prod^J_{j=1} z_j^{S_j} \Bigg] = G_S \Bigg( \sum^J_{j=1} p_j z_j  \Bigg)
\end{align*}
where $S$ is a univariate random variable within Sundt and Jewell's family \parencite{sundt1981further, willmot1988sundt} and $p_j \geq 0$ with $\sum_j p_j=1$. This corresponds to a univariate batch, with each incoming customer independently assigned an entry point $j \in \{1, \dots, J\}$. For notational convenience, we set
\begin{align*}
    q_k(t) = \sum^{J}_{j=1} p_j \cdot q_k^j(t)
\end{align*}
to be the probability that a customer is situated in queue $k$ time $t$ after their arrival into the network. In the corollaries below, we state explicit formulae for binomial, Poisson, negative binomial and logarithmic batch sizes.

\begin{corollary} \label{corollary::binomial}
Suppose $S \sim \text{Binomial}(N, \alpha)$. Then for $v \in \{2, \dots, J \}$ and $n_v \geq 1$
\begin{align*}
    P \big[ & \mathbf{N}(t) = (n_1, \dots, n_v, 0, \dots 0 ) \big]\\
    & = \sum^{n_1}_{i_1 = 0} \dots \sum^{n_{v-1}}_{i_{v-1} = 0} \sum^{n_{v}}_{i_{v} = 1} \Bigg[ \frac{i_v}{n_v} \cdot  P\big[ \mathbf{N}(t) = ( n_1-i_1, \dots, n_v-i_v, 0, \dots 0) \big] \cdot \mathbbm{1}_{\sum^v_{j=1} i_j \leq N} \\
    & \qquad \qquad \cdot \int^t_0 \lambda(\tau) \frac{N!}{\big( N - \sum^v_{k=1} i_k \big)!} \Big( 1 - \alpha \sum^v_{k=1} q_k(t-\tau) \Big)^{N-\sum^v_{k=1} i_k} \prod^v_{k=1} \frac{  ( \alpha q_k(t-\tau))^{i_k}}{i_k!} d \tau \Bigg].
\end{align*}
\end{corollary}

\begin{corollary}
Suppose $S \sim \text{Poisson}(\mu)$. Then for $v \in \{2, \dots, J \}$ and $n_v \geq 1$
\begin{align*}
    P \big[ & \mathbf{N}(t) = (n_1 , \dots, n_v, 0, \dots 0 ) \big]\\
    & = \sum^{n_1}_{i_1 = 0} \dots \sum^{n_{v-1}}_{i_{v-1} = 0} \sum^{n_{v}}_{i_{v} = 1} \Bigg[ \frac{i_v}{n_v} \cdot  P \big[ \mathbf{N}(t) = (n_1-i_1, \dots, n_v-i_v, 0, \dots, 0 ) \big] \cdot \int^t_0 \lambda(\tau) \bigg[ \prod^v_{k=1} \frac{(q_k(t-\tau) \mu)^{i_k}}{i_k!} \bigg] e^{-\mu}  d \tau \Bigg].
\end{align*}
\end{corollary}

\begin{corollary}
Suppose $S$ follows a negative binomial distribution with PGF
\begin{align*}
    G(z) := \EX \big[ z^S ] = \frac{1}{\big[ 1 + \nu(1-z) \big]^r}.
\end{align*}
Then for $v \in \{2, \dots, J \}$ and $n_v \geq 1$
\begin{align*}
    P \big[ & \mathbf{N}(t) = (n_1, \dots, n_v, 0, \dots 0 ) \big]\\
    & = \sum^{n_1}_{i_1 = 0} \dots \sum^{n_{v-1}}_{i_{v-1} = 0} \sum^{n_{v}}_{i_{v} = 1} \Bigg[ \frac{i_v}{n_v} \cdot  P \big[ \mathbf{N}(t) = (n_1-i_1, \dots, n_v-i_v, 0, \dots 0) \big] \\
    & \qquad \qquad \cdot \int^t_0 \lambda(\tau) \frac{\Gamma \big(r + \sum^v_{k=1} i_k \big)}{\Gamma(r) \prod^v_{k=1} i_k!} \frac{\nu^{\sum^v_{k=1} i_k}}{\big[ 1 + \nu \sum^J_{k=1} q_k(t-\tau) \big]^{r+\sum^v_{k=1} i_k}} \prod^v_{k=1} q_k(t-\tau)^{i_k} d \tau \Bigg].
\end{align*}
\end{corollary}

\begin{corollary} \label{corollary::log}
Suppose $S \sim \text{Log}(\rho)$, with PMF
\begin{align*}
    P(S=n) = -\frac{1}{\log(1-\rho)} \frac{\rho^k}{k}, \, k \in \mathbbm{Z}^{+}.
\end{align*}
Then for $v \in \{2, \dots, J \}$ and $n_v \geq 1$ 
\begin{align*}
    P \big[ & \mathbf{N}(t) = (n_1, \dots, n_v, 0, \dots 0 ) \big]\\
    & = \sum^{n_1}_{i_1 = 0} \dots \sum^{n_{v-1}}_{i_{v-1} = 0} \sum^{n_{v}}_{i_{v} = 1} \Bigg[ \frac{i_v}{n_v} \cdot  P \big[ \mathbf{N}(t) = (n_1-i_1, \dots, n_v-i_v, 0, \dots 0) \big] \\
    & \qquad \qquad \cdot \int^t_0 -\lambda(\tau) \frac{\rho^{\sum^v_{k=1} i_k}}{\log(1-\rho)} \bigg( \prod^v_{k=1} \frac{  q_k(t-\tau)^{i_k}}{i_k!} \bigg) \frac{\big(-1 + \sum^v_{v=1} i_k \big)!}{\big[ 1 - \rho \big( 1 - \sum^J_{k=1} q_k(t-\tau) \big) \big]^{\sum^v_{k=1} i_k}}  d \tau \Bigg].
\end{align*}
\end{corollary}

\subsubsection{Constant batch sizes}

Suppose $\mathbf{S} = \mathbf{s}$ is constant. Then $\mathbf{C}(t-\tau)$ is the sum of independent, but non-identical categorical random variables, that is, $\mathbf{C}(t-\tau)$ follows a Poisson multinomial distribution \parencite{lin2022poisson}. Using the exact DFT-CT method proposed by \textcite{lin2022poisson}, we can recover the PMF for $\mathbf{C}(t-\tau)$, which can then be plugged into the recurrence relation (\ref{eq::transient_recurrence_relation}) to compute the transient occupancy distribution.

\subsubsection{A model of superinfection and hypnozoite dynamics for vivax malaria}

In \textcite{mehra2022hypnozoite}, we constructed an open network of infinite server queues with exponential service times to model within-host superinfection and hypnozoite dynamics for \textit{Plasmodium vivax} malaria. A schematic of the queueing network, with nodes labelled $\{1, \dots, k, NL, A, D, C, P, PC \}$, is replicated from Figure 3 of \textcite{mehra2022hypnozoite} below.\\

\begin{figure}[ht]
    \centering
    \includegraphics[width=0.9\textwidth]{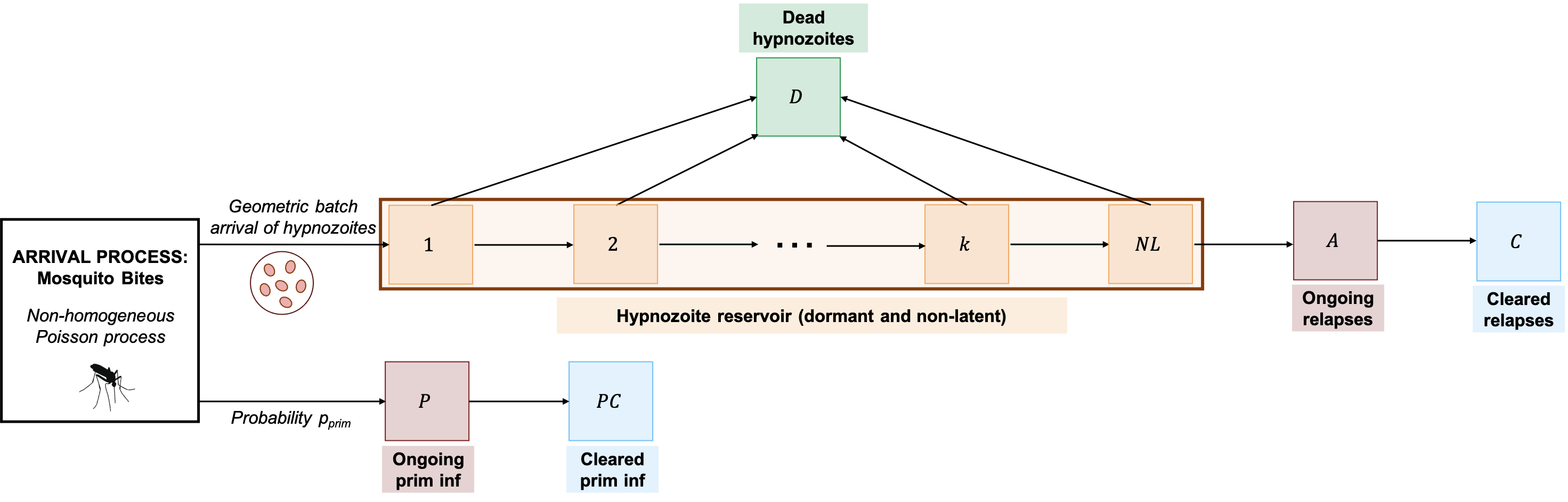}
    \caption{Model schematic replicated from Figure 3 of \textcite{mehra2022hypnozoite}}
    \label{fig:vivax_queue}
\end{figure}

The batch takes the form of a single arrival into queue $P$ with probability $p_\text{prim}$ (representing a primary blood-stage infection) and a geometrically-distributed arrival into queue $1$ (representing the establishment of dormant hypnozoites). Each hypnozoite then undergoes a compartmental process, accounting for death (queue $D$), or progression through successive latency compartments (queues $2, \dots, k$) and subsequent activation giving rise to a blood-stage relapse (queue $A$). We further account for the clearance of each primary infection (queue $P$ to $PC$) and relapse (queue $A$ to $C$). Analysis of the queueing network allows for the characterisation of quantities of epidemiological interest, with applications to population-level transmission modelling \parencite{mehra2022hybrid}.

\section{Conclusion}
We have extended results for single infinite-server queues with batch arrivals and open networks of infinite-server queues with single arrivals to a general network of infinite server queues with multivariate batch arrivals arriving according to a non-homogeneous Poisson process. Theorem 3.1 gives an expression for the PGF of the transient queue length distribution, Theorem 4.1 a necessary and sufficient condition for ergodicity and Theorem 5.1 a recurrence relation for the multivariate probability mass function. 

\section{Acknowledgements}
The authors would like to acknowledge funding from the Australian Research Council through Laureate Fellowship FL130100039.

\newpage
\printbibliography

@article{mehra2022hypnozoite,
  title={Hypnozoite dynamics for Plasmodium vivax malaria: The epidemiological effects of radical cure},
  author={Mehra, Somya and Stadler, Eva and Khoury, David and McCaw, James M and Flegg, Jennifer A},
  journal={Journal of Theoretical Biology},
  volume={537},
  pages={111014},
  year={2022},
  publisher={Elsevier}
}

@article{harrison1981note,
  title={A note on networks of infinite-server queues},
  author={Harrison, J Michael and Lemoine, Austin J},
  journal={Journal of Applied Probability},
  pages={561--567},
  year={1981},
  publisher={JSTOR}
}

@book{tijms2003first,
  title={A first course in stochastic models},
  author={Tijms, Henk C},
  year={2003},
  publisher={John Wiley and sons}
}

@article{brown1969some,
  title={Some results for infinite server Poisson queues},
  author={Brown, Mark and Ross, Sheldon M},
  journal={Journal of Applied Probability},
  volume={6},
  number={3},
  pages={604--611},
  year={1969},
  publisher={Cambridge University Press}
}

@article{miatto2019recursive,
  title={Recursive Multivariate Derivatives of $e^{f(X_1,..., X_n)}$ of Arbitrary Order},
  author={Miatto, Filippo M},
  journal={arXiv preprint arXiv:1911.11722 [cs, math]},
  year={2019}
}

@techreport{keilson1990networks,
  title={Networks of Non-homogeneous M/G/$\infty$ Systems},
  author={Keilson, Julian and Servi, Les D},
  year={1990},
  institution={Massachusetts Institute of Technology},
  publisher={Massachusetts Institute of Technology, Operations Research Center}
}

@article{mccalla2002time,
  title={A time-dependent queueing-network model to describe the life-cycle dynamics of private-line telecommunication services},
  author={McCalla, Clement and Whitt, Ward},
  journal={Telecommunication Systems},
  volume={19},
  number={1},
  pages={9--38},
  year={2002},
  publisher={Springer}
}

@article{kella1999linear,
  title={Linear stochastic fluid networks},
  author={Kella, Offer and Whitt, Ward},
  journal={Journal of Applied Probability},
  volume={36},
  number={1},
  pages={244--260},
  year={1999},
  publisher={Cambridge University Press}
}

@article{panjer1981recursive,
  title={Recursive evaluation of a family of compound distributions},
  author={Panjer, Harry H},
  journal={ASTIN Bulletin: The Journal of the IAA},
  volume={12},
  number={1},
  pages={22--26},
  year={1981},
  publisher={Cambridge University Press}
}

@article{reynolds1968some,
  title={Some results for the bulk-arrival infinite-server Poisson queue},
  author={Reynolds, John F},
  journal={Operations Research},
  volume={16},
  number={1},
  pages={186--189},
  year={1968},
  publisher={INFORMS}
}

@article{lin2022poisson,
  title={The Poisson Multinomial Distribution and its Applications in Voting Theory, Ecological Inference, and Machine Learning},
  author={Lin, Zhengzhi and Wang, Yueyao and Hong, Yili},
  journal={arXiv preprint arXiv:2201.04237},
  year={2022}
}

@article{mehra2022hybrid,
  title={A hybrid transmission model for Plasmodium vivax accounting for superinfection, immunity and the hypnozoite reservoir},
  author={Mehra, Somya and Taylor, Peter G and McCaw, James M and Flegg, Jennifer A},
  journal={arXiv preprint arXiv:2208.10403},
  year={2022}
}

@book{korevaar2017several,
  title={Several complex variables},
  author={Korevaar, Jacob and Wiegerinck, Jan},
  year={2017},
  publisher={Korteweg-de Vries Institute for Mathematics}
}

@article{massey1993networks,
  title={Networks of infinite-server queues with nonstationary Poisson input},
  author={Massey, William A and Whitt, Ward},
  journal={Queueing Systems},
  volume={13},
  number={1},
  pages={183--250},
  year={1993},
  publisher={Springer}
}

@article{chatterjee1989non,
  title={On the non-homogeneous service system MX/G/$\infty$},
  author={Chatterjee, U and Mukherjee, SP},
  journal={European Journal of Operational Research},
  volume={38},
  number={2},
  pages={202--207},
  year={1989},
  publisher={Elsevier}
}

@article{liu1991thegr,
  title={The $GR^{X_n}/G_n/\infty$ system: system size},
  author={Liu, Liming and Templeton, JGC},
  journal={Queueing Systems},
  volume={8},
  number={1},
  pages={323--356},
  year={1991},
  publisher={Springer}
}

@article{daw2019distributions,
  title={On the distributions of infinite server queues with batch arrivals},
  author={Daw, Andrew and Pender, Jamol},
  journal={Queueing Systems},
  volume={91},
  number={3},
  pages={367--401},
  year={2019},
  publisher={Springer}
}

@book{comtet2012advanced,
  title={Advanced Combinatorics: The art of finite and infinite expansions},
  author={Comtet, Louis},
  year={2012},
  publisher={Springer Science \& Business Media}
}

@article{willmot1988sundt,
  title={Sundt and Jewell's family of discrete distributions},
  author={Willmot, Gordon},
  journal={ASTIN Bulletin: The Journal of the IAA},
  volume={18},
  number={1},
  pages={17--29},
  year={1988},
  publisher={Cambridge University Press}
}

@article{schumann2019multivariate,
  title={Multivariate bell polynomials and derivatives of composed functions},
  author={Schumann, Aidan},
  journal={arXiv preprint arXiv:1903.03899},
  year={2019}
}

@article{sundt1981further,
  title={Further results on recursive evaluation of compound distributions},
  author={Sundt, Bj{\o}rn and Jewell, William S},
  journal={ASTIN Bulletin: The Journal of the IAA},
  volume={12},
  number={1},
  pages={27--39},
  year={1981},
  publisher={Cambridge University Press}
}

@incollection{willmot2002transient,
  title={Transient analysis of some infinite server queues},
  author={Willmot, Gordon E and Drekic, Steve},
  booktitle={Recent Advances in Statistical Methods},
  pages={329--338},
  year={2002},
  publisher={World Scientific}
}

@article{shanbhag1966infinite,
  title={On infinite server queues with batch arrivals},
  author={Shanbhag, DN},
  journal={Journal of Applied Probability},
  volume={3},
  number={1},
  pages={274--279},
  year={1966},
  publisher={Cambridge University Press}
}

@misc{abramowitz1988handbook,
  title={Handbook of mathematical functions with formulas, graphs, and mathematical tables},
  author={Abramowitz, Milton and Stegun, Irene A and Romer, Robert H},
  year={1988},
  publisher={American Association of Physics Teachers}
}

@article{willmot2001transient,
  title={On the transient analysis of the $M^X/M/\infty$ queue},
  author={Willmot, Gordon E and Drekic, Steve},
  journal={Operations Research Letters},
  volume={28},
  number={3},
  pages={137--142},
  year={2001},
  publisher={Elsevier}
}

@article{cong1994mx,
  title={On the $M^X/G/\infty$ queue with heterogeneous customers in a batch},
  author={Cong, Tang Dac},
  journal={Journal of Applied Probability},
  pages={280--286},
  year={1994},
  publisher={JSTOR}
}

@article{willmot2009time,
  title={Time-dependent analysis of some infinite server queues with bulk Poisson arrivals},
  author={Willmot, Gordon E and Drekic, Steve},
  journal={INFOR: Information Systems and Operational Research},
  volume={47},
  number={4},
  pages={297--303},
  year={2009},
  publisher={Taylor \& Francis}
}

@inproceedings{yajima2016stability,
  title={The stability condition of BMAP/M/$\infty$ queues},
  author={Yajima, Moeko and Phung-Duc, Tuan and Masuyama, Hiroyuki},
  booktitle={Proceedings of the 11th International Conference on Queueing Theory and Network Applications},
  pages={1--6},
  year={2016}
}

@book{whittaker2020course,
  title={A course of modern analysis},
  author={Whittaker, Edmund Taylor and Watson, George Neville},
  year={2020},
  publisher={Courier Dover Publications}
}

@book{jeffrey2007table,
  title={Table of integrals, series, and products},
  author={Jeffrey, Alan and Zwillinger, Daniel},
  year={2007},
  publisher={Elsevier}
}

\end{document}